\documentclass[12pt,a4paper,reqno]{amsart}
\usepackage{amsfonts}
\usepackage{lscape}
\usepackage{amsmath}
\usepackage{amssymb}
\usepackage{geometry}
\usepackage{latexsym}
\usepackage{afterpage}
\usepackage{amstext}
\usepackage{amscd}
\usepackage{setspace}
\usepackage{caption}
\usepackage{graphicx}
\usepackage{enumerate}
\usepackage{bbm}
\usepackage{hyperref}

\usepackage{color}
\usepackage{xcolor}

\newtheorem{theorem}{Theorem}[section]
\numberwithin{equation}{section}

\newtheorem{definition}[theorem]{Definition}
\newtheorem{example}[theorem]{Example}

\newtheorem{lemma}[theorem]{Lemma}

\newtheorem{proposition}[theorem]{Proposition}
\newtheorem{remark}[theorem]{Remark}

\newtheorem*{Theorem*}{Theorem}
\newtheorem{maintheorem}{Theorem}

\newtheorem{maincorollary}[maintheorem]{Corollary}

\newcommand{\cmc}{\begin{maincorollary}}
	\newcommand{\fmc}{\end{maincorollary}}

\renewcommand{\varepsilon}{\epsilon}

\newcommand{\m}{m}

\DeclareMathOperator{\Cor}{Cor}

\title[Decay of Correlations via WGM for non-H\"{o}lder observables]
{Decay of Correlations via induced Weak Gibbs Markov maps for non-H\"{o}lder observables}

\author[A. Ullah]{Asad Ullah}
\address{Centro de Matem\'atica e Aplica\c{c}\~oes (CMA-UBI), Universidade da Beira
	Interior, Rua Marqu\^es d'\'Avila e Bolama, 6201-001, Covilh\~a, Portugal.}
\email{asad.ullah@ubi.pt}

\author[H. Vilarinho]{Helder Vilarinho}
\email{helder@ubi.pt}

\date{\today}



\thanks{A. Ullah and H. vilarinho were partially supported by Funda\c{c}\~ao para a Ci\^encia e a Tecnologia (FCT) 
	through Centro de Matemática e Aplicações (CMA-UBI), Universidade da Beira Interior, under the project
	UIDB/00212/2020. A. Ullah was also supported by FCT under the grant number UI/BD/150796/2020.}

\begin{document}
\begin{abstract}
We extend the results of \cite{our 1st article} by considering larger classes of observables. More precisely, we obtain estimates on the decay of correlations, Central Limit Theorem and Large Deviations for dynamical systems having an induced weak Gibbs Markov map, for larger classes of observables with weaker regularity than H\"{o}lder.
\end{abstract}

 \maketitle 

\section{Introduction}
 Statistical properties for H\"{o}lder observables are well understood for a large variety of dynamical systems. This includes the decay of correlations, central limit theorem (CLT), large deviations, almost sure invariance principle, etc. For instance, in \cite{Y98, Y99} the decay of correlation and CLT are obtained for nonuniformly hyperbolic systems. The large deviations and almost sure invariance principles for nonuniformly expanding systems are discussed in \cite{Invariance Principle via full return induced GM, Large deviation via full return induced GM}. We also refer to \cite{D.Corrlation for Henon map, Decay of correlation via full return induced GM map, Billiards with poly-mixing, Statistical properties via full return induced GM map, Gouezel reproducing full return GM} where these statistical properties for distinct classes of dynamical systems are analysed. Note that in all of the above references, observables are assumed to be H\"{o}lder.

There are several works dealing with classes of observables strictly larger than H\"{o}lder. For example, the mixing rate of the equilibrium state of one-sided shift is discussed in \cite{BFG} for non-H\"{o}lder potential of summable variations, while the results for two-sided subshift of finite type are discussed in \cite{ Policott}.
In addition, the polynomial decay of correlation and CLT are obtained in \cite{FL} for the equilibrium state of a one-sided shift map on two symbols with non-H\"{o}lder potential. We also would like to mention some work beyond shift spaces. In \cite{PY} a class of observables which contains all piecewise Lipschitz functions is considered, in order to get the decay of correlation for
certain non-uniformly expanding systems. Estimates for the rates of mixing for observables with weaker regularity than H\"{o}lder  are given in~\cite{D}  for one-dimensional expanding Lorenz-like map. We refer to \cite{R,Z} for invertible maps with non-H\"{o}lder observables. More generally, in \cite{Lynch paper}, decay of correlations and CLT are obtained for those dynamical systems that admit an induced full branch map, referred to as Gibbs Markov map, for much larger classes of observables. More precisely, the results obtained in \cite{Lynch paper} are an extension of \cite{Y99} in the sense of considering bigger classes of observables. 

In \cite{our 1st article}, the authors generalised the results of \cite{Y99} under weaker assumptions on the dynamical system, where the induced map was not necessarily full branch, that we refer to as a weak Gibbs Markov (WGM) map; see Definition \ref{def:WGM}. 
 Our goal in the present work is to extend the results of \cite{our 1st article} by taking the bigger classes of observables considered in \cite{Lynch paper}. Our approach is to consider a mixing tower built over an induced WGM map, and then transfer the information obtained to the original dynamical system. We make an extension of the coupling arguments used in \cite{our 1st article} in order to obtain the decay of correlations for the tower system for larger classes of observables. In addition, we obtain CLT and large deviation results as an immediate corollary.

This article is organised as follows:
In Section \ref{V.relvant def and main results}, we give the necessary definitions and state the main results. In Section \ref{V.se:dc on tower}, we obtain the decay of correlations for the tower map and transfer them to the original dynamical system.
 
\section{Preliminaries and statement of main results}\label{V.relvant def and main results}
Consider a finite measure space $(\Delta_0,\mathcal A,\m)$, with $m(\Delta_{0})>0$, a measurable map $ F: \Delta_{0}\rightarrow \Delta_{0}$ and an ($\m$ mod 0) finite or countable partition $\mathcal{P}_{0}$ of $ \Delta_{0} $ into invertibility  domains of $F$, that is, $F$ is a bijection from each $\omega\in \mathcal{P}_{0}$ to $F(\omega)$, with measurable inverse.
For each $n\geq0$ set $ F^{-n}(\mathcal{P}_{0})=\{F^{-n}(\omega): \omega\in \mathcal{P}_{0}\}$. We define
\begin{equation*}
	\mathcal{P}_{0}^{n}=\bigvee_{i=0}^{n-1}F^{-i}(\mathcal{P}_{0}), \text{ for $n\geq 1$},	\,\,\text { and } \,\,\mathcal P_0^\infty=\bigvee_{i=0}^{\infty}F^{-i}(\mathcal{P}_{0}).
\end{equation*} 
We call the sequence $(\mathcal P_0^n)_{n}$ a \textit{basis of  $\Delta_{0}$} if it generates $\mathcal{A}$ ($\m$ mod 0) and $\mathcal P_0^\infty$ is the partition into single points.

\begin{definition}\label{def:WGM}
	We say that $F: \Delta_{0}\rightarrow \Delta_{0}$ is a weak Gibbs Markov (WGM) map, with respect to the  partition $\mathcal{P}_{0}$, if the following hold:
	\begin{enumerate}[W$1$)]
		\item \label{def:WGM:Markov} Markov: $F$ maps each $\omega\in \mathcal{P}_{0}$ bijectively to an ($\m$ mod 0) union of elements of $\mathcal{P}_{0}$.
		
		\item Separability: the sequence $(\mathcal P_0^n)_{n}$ is a basis of  $\Delta_{0}$.

		\item Nonsingular: there exists a strictly positive measurable function $J_{F}$ defined on $\Delta_0$ such that for each $A\subset\omega\in \mathcal{P}_{0}$,
		\begin{equation*}
			m(F(A))=\int_{A}J_{F}dm.
		\end{equation*}
		\item Gibbs: there exist $C_{F}>0$ and $0<\beta<1$ such that for all  $\omega\in \mathcal{P}_{0}$ and $x,y\in \omega$,
		\begin{equation*}
			\log\frac{J_{F}(x)}{J_{F}(y)}\leq C_{F}\beta^{s(F(x), F(y))},
		\end{equation*}
		where 
		$$s(x,y)=min\{n\geq0:F^{n}(x) \mbox{ and } F^{n}(y) \mbox{ lie in distinct elements of } \mathcal{P}_{0} \}.$$
		\item Long branches: there exists $ \delta_{0}>0 $ such that $ m(F(\omega))\geq \delta_{0} $, for all $ \omega\in \mathcal{P}_{0}$.
	\end{enumerate}
\end{definition}
The term \emph{weak} in the above definition refers that in W\ref{def:WGM:Markov}) we do not require \textit{full branch} (that is, $F(\omega)=\Delta_0$ for all $\omega\in\mathcal P_0$). In the case of full branch, $F$ is called a \textit{Gibbs Markov} map. This terminology follows \cite{Alves Book}.
\medskip

Consider a measure space $(M,\mathcal B,\m)$, a measurable map $f: M\rightarrow M$ and $\Delta_{0}\subset M$ with $m(\Delta_{0})>0$. For simplicity we denote the restriction of $\m$ to $\Delta_0$ also by $\m$. We say that $ F:\Delta_{0}\rightarrow \Delta_{0} $ is an \emph{induced map for $f$} if there exists
a countable ($\m$ mod 0) partition $\mathcal{P}_{0}$ of $\Delta_{0}$ and a measurable function $R:\Delta_{0}\rightarrow \mathbb{N}$, constant on each element of $\mathcal{P}_{0}$, such that 
$$F|_{\omega}=f^{R(\omega)}|_{\omega}.$$ 
We formally denote the induced map $F$ by $f^{R}$.
We say that an induced map $f^R\colon\Delta_{0}\rightarrow\Delta_{0}$ is  \emph{aperiodic} if for all  $\omega_{1}, \omega_{2}\in\mathcal{P}_{0}$ there exists  $k_{0}\in \mathbb{N}$ such that  for all  $n\geq k_{0}$,
\begin{equation*}
	m(\omega_{1}\cap (f^R)^{-n}(\omega_{2}))>0.
\end{equation*}
We say that an induced map $f^R$ has a \emph{coprime block} if there exist $N\geq2$ and $ \omega_{1},\omega_{2},...,\omega_{N}\in\mathcal{P}_0 $
such that $\gcd\{R(\omega_i)\}_i=1$ and for all $i=1,\ldots,N$,
\begin{equation*}
	f^{R}(\omega_{i})\supseteq\omega_{1}\cup\omega_{2}\cup\cdots\cup\omega_{N}\,\,\,(\m\text{ mod } 0).
\end{equation*}

Assume moreover that $M$ is a metric space with metric $d$. 
We say that an induced map $f^{R}$ is \emph{expanding} if there are $0 < \beta < 1$ and  $C>0$ such that, for all $\omega \in \mathcal{P}_0$ and $x,y \in \omega$
\begin{itemize}
	\item[i)] $d(f^{R}(x), f^{R}(y)) \leq C \beta^{s(x,y)}$,
	\item[ii)] $d(f^{j}(x), f^{j}(y)) \leq Cd(f^{R}(x), f^{R}(y)),$ for all $0 \leq j \leq R.$
\end{itemize}
 For a given function $\varphi: M \rightarrow \mathbb{R}$ and $\epsilon>0$ we set
$$
\mathcal{R}_{\epsilon}(\varphi):=\sup \{|\varphi(x)-\varphi(y)|: d(x, y) \leq \epsilon\}.
$$
We consider the following classes of observables which were defined in \cite{Lynch paper}:
\begin{itemize}
	\item  $(R 1, \tau) =\{ \varphi : \mathcal{R}_{\epsilon}(\varphi)=\mathcal{O}(\epsilon^\tau)\}$, $\tau \in(0,1)$.

\item  $(R 2, \tau)=\{\varphi :\mathcal{R}_{\epsilon}(\varphi)=\mathcal{O}(\exp \{-|\log \epsilon|^\tau\})\}$, $\tau \in(0,1)$.

\item  $(R 3, \tau)=\{\varphi : \mathcal{R}_{\epsilon}(\varphi)=\mathcal{O}(\exp \{-(\log |\log \epsilon|)^\tau\})\}$, $  \tau>1 $.

\item  $(R 4, \tau)=\{ \varphi : \mathcal{R}_{\epsilon}(\varphi)=\mathcal{O}(|\log \epsilon|^{-\tau})\}$, $ \tau>1 $.
\end{itemize}
The correlation sequence of two observable functions  $\varphi \in\mathcal{H}\subseteq L^\infty(M,\m)$ and $\psi \in L^\infty(M,\m)$, with respect to an $f$-invariant probability $\mu$, is defined by
\begin{equation*}
	\Cor_{\mu}(\varphi, \psi \circ f^{n})=\left|\int\varphi(\psi\circ f^{n}) d\mu-\int\varphi d\mu\int\psi d\mu\right|.
\end{equation*}

\begin{maintheorem}\label{V.teo:main}
Consider a measure space $(M,\mathcal B,\m)$, endowed with some metric, and a measurable map $f: M\rightarrow M$ satisfying $f_{*}m \ll m$. Let $f^{R}: \Delta_{0} \rightarrow \Delta_{0}$ be an aperiodic induced WGM expanding map with a coprime block, $R \in L^{1}(\m)$ and let $\mu$ be the unique ergodic $f$-invariant probability measure  $\mu$ such that $\mu \ll m$ and $\mu(\Delta_{0})>0$. Then,
	\begin{enumerate}
		\item If $m\{ R >n \} \leq C n^{-a}$ for some $C>0$ and $a >1$, given $ \kappa>0$, there is $0<\zeta<1$ such that for all $\varphi \in (R 4, \tau)$, with $ \tau>\frac{2}{\zeta}$ and $ \mathcal{R}_{\infty}(\varphi)< \kappa$, and for all $\psi \in L^\infty(M,\m)$ we have
		\begin{itemize}
		\item[i)] if $\tau=\frac{a+1}{\zeta} $, there is some $C{'} > 0$ such that
		$$\Cor_{\mu}(\varphi, \psi \circ f^{n}) \leq C^{'} \left(n^{1-a} \log n\right);$$
	\item[ii)] otherwise, there is some $C{'} > 0$ such that
	$$ \Cor_{\mu}(\varphi, \psi \circ f^{n}) \leq C^{'}\max \left\{n^{1-a}, n^{2-\zeta \tau}\right\}.$$
\end{itemize}
\item If $m\{ R >n \} \leq C e^{-cn}$ for some $C,c>0$, then 
	\begin{itemize}
		\item[i)] for all $\varphi \in (R 1, \tau)$ and $\psi \in L^\infty(M,\m)$ there are  $C{'},c' > 0$ such that
		$$\Cor_{\mu}(\varphi, \psi \circ f^{n}) \leq C^{\prime}e^{-c^{\prime}n};$$
	\item[ii)] for all $\varphi \in (R 2, \tau)$ and $\psi \in L^\infty(M,\m)$ there is  $C{'} > 0$ such that, for every $\tau^{\prime}<\tau$,
	$$\Cor_{\mu}(\varphi, \psi \circ f^{n}) \leq C^{\prime}e^{-n^{\tau^{\prime}}};$$
	\item[iii)] for all $\varphi \in (R3, \tau)$ and $\psi \in L^\infty(M,\m)$ there is  $C{'} > 0$ such that, for every $\tau^{\prime}<\tau$,
	$$\Cor_{\mu}(\varphi, \psi \circ f^{n}) \leq C^{\prime}e^{-(\log n)^{\tau^{\prime}}};$$
\item[iv)] for given $\kappa>0 $ there is $ 0<\zeta<1 $ such that for all $\varphi \in (R 4, \tau)$, with $ \tau>\frac{1}{\zeta}$ and $ \mathcal{R}_{\infty}(\varphi)< \kappa$, and for all $\psi \in L^\infty(M,\m),$ there is  $C{'} > 0$ such that
$$\Cor_{\mu}(\varphi, \psi \circ f^{n}) \leq C^{\prime}n^{1-\zeta\tau}.$$
\end{itemize}
\end{enumerate}
\end{maintheorem}
We remark that the existence of such measure $ \mu $ is proved in \cite{our 1st article}. The proof of remaining parts of Theorem~\ref{V.teo:main} can be deduced from Proposition \ref{V.Convergence to equilibrium} (see Remark~\ref{main thm.p}).
\medskip

Let $ \mu $ be an ergodic $ f$-invariant probability measure. We say that an observable $ \varphi: M \rightarrow \mathbb{R} $ with $ \int\varphi d\mu=0 $ satisfies the \textit{CLT} if $ \frac{1}{\sqrt{n}}\sum_{i=0}^{n-1} \varphi \circ f^{i} $ converges in law (or in distribution) to a normal distribution $ \mathcal{N}(0,\sigma)$, for some $ \sigma>0.$ We may also consider observables of non-zero expectation by replacing $ \varphi $ with $ \varphi-\int \varphi d\mu$. In this situation an observable $ \varphi $ satisfies the CLT if there exists $ \sigma>0 $ such that for every interval $ J\subset \mathbb{R} $,
\begin{equation*}
	\mu \left\{ x : \frac{1}{\sqrt{n}}\sum_{i=0}^{n-1}\left(\varphi( f^{i}(x))- \int \varphi d\mu \right)\in J \right\} \rightarrow \frac{1}{\sigma \sqrt{2\pi}}\int_{J}e^{-\frac{t^2}{2\sigma^2}}dt, \text{ as } n\rightarrow \infty.
\end{equation*}
A function $ \varphi $ is \textit{coboundary} if there exists a measurable function $ g $   such that $ \varphi\circ f=g\circ f-g$.  
\begin{maincorollary}[Central Limit Theorem]\label{V.CLT Main Cror}
	Consider a measure space $(M,\mathcal B,\m)$, endowed with some metric, and a measurable map $f: M\rightarrow M$ satisfying $f_{*}m \ll m$. Let $f^{R}: \Delta_{0} \rightarrow \Delta_{0}$ be an aperiodic induced WGM expanding map with a coprime block such that $R \in L^{1}(\m)$, and let $\mu$ be the unique ergodic $f$-invariant probability measure  $\mu$ such that $\mu \ll m$ and $\mu(\Delta_{0})>0$. 
	If $m\{ R >n \} \leq C n^{-a}$ for some $C>0$ and $a >2$, then the CLT is satisfied for all $\varphi \in (R 4, \tau)$, some $ \tau>0$ sufficiently large if and only if $\varphi$ is not coboundary.
\end{maincorollary}
 The proof of Corollary~\ref{V.CLT Main Cror} is essentially the same as that given in \cite[Corollary~B]{our 1st article}, and we omit it here. 

\medskip

Given  $ \epsilon>0 $ we define the \textit{large deviation at time $ n $} of the time average of an observable $ \varphi\colon M\to\mathbb R $ from the spatial average as
\begin{equation*}
	LD_{\mu}(\varphi, \epsilon, n):	=\mu\left(\left|\frac{1}{n}\sum_{i=0}^{n-1} \varphi\circ f^i - \int \varphi d\mu\right|> \epsilon \right).
\end{equation*} 
By combining Theorem~\ref{V.teo:main} and \cite{statistical properties implies full return GM, LD with improved polynomial rate}, we have an immediate corollary.  
\begin{maincorollary}[Large Deviations]\label{V.corol:LD}
	Consider a measure space $(M,\mathcal B,\m)$, endowed with some metric, and a measurable map $f: M\rightarrow M$ satisfying $f_{*}m \ll m$. Let $f^{R}: \Delta_{0} \rightarrow \Delta_{0}$ be an aperiodic induced WGM expanding map with a coprime block such that $R \in L^{1}(\m)$, and let $\mu$ be the unique ergodic $f$-invariant probability measure  $\mu$ such that $\mu \ll m$ and $\mu(\Delta_{0})>0$. Then,
	\begin{enumerate}
		\item If $m\{ R >n \} \leq C n^{-a}$ for some $C>0$ and $a >1$, given $  \kappa>0 $ there is $ 0<\zeta<1 $, such that for all $ \epsilon>0 $ and  $\varphi \in (R 4, \tau)$, with $ \frac{2}{\zeta}<\tau\neq\frac{a+1}{\zeta}$ and $ \mathcal{R}_{\infty}(\varphi)< \kappa$, there is some $C^{'}=C^{'}(\epsilon, \varphi) > 0$ such that
		$$
		LD_{\mu}(\varphi, \epsilon, n) \leq C^{'}\max\{n^{-a +1}, n^{2-\zeta\tau}\}.
		$$
		\item If $m\{ R >n \} \leq C e^{-cn}$ for some $C,c>0$, then 
		\begin{itemize}
	\item[i)] for all $ \epsilon>0 $ and $\varphi \in (R 1, \tau)$, there is some $C{'}=C^{'}(\epsilon, \varphi)> 0$ and $c'=c'(c, \varphi, \epsilon, \tau) $ such that
		$$
		LD_{\mu}(\varphi, \epsilon, n) \leq C^{'} e^{-c{'}n^{\frac{1}{3}}};
		$$
	\item[ii)] for all $ \epsilon>0 $ and $\varphi \in (R 2, \tau)$, there is some $C{'}=C^{'}(\epsilon, \varphi)> 0$ and $c'=c'( \varphi, \epsilon)$ such that, for every $\tau^{\prime}<\tau$,
	$$
	LD_{\mu}(\varphi, \epsilon, n) \leq C^{'} e^{-c{'}n^{\frac{\tau^{\prime}}{\tau^{\prime}+2}}};
	$$
\item[iii)] for given $\kappa>0$ there is $ 0<\zeta<1 $ such that for all $ \epsilon>0 $ and $\varphi \in (R4, \tau)$ with $ \tau>\frac{1}{\zeta}$ and $ \mathcal{R}_{\infty}(\varphi)< \kappa$, there is some $C{'}=C^{'}(\epsilon, \varphi)> 0$, such that
$$
LD_{\mu}(\varphi, \epsilon, n) \leq C^{'} n^{1-\zeta\tau}.
$$
\end{itemize}
\end{enumerate}
\end{maincorollary}

\begin{example}
	Let $M=S^1\times[0,1]$ and let $\m$ denote the Lebesgue measure on $M$. We consider the map $f:M\rightarrow M$ introduced in~\cite{Gouezel skew product} defined by
	$$
	f(\theta, x)=\left(F(\theta), f_{\alpha(\theta)}(x)\right),
	$$
	where $F(\theta)=4\theta$,
	$$
	f_\alpha(\theta)(x)= \begin{cases}x\left(1+2^{\alpha(\theta)} x^{\alpha(\theta)}\right) & \text { if } 0 \leqslant x \leqslant \frac{1}{2} \\ 2 x-1 & \text { if } \frac{1}{2}<x \leqslant 1,\end{cases}
	$$
	and  $\alpha: S^1 \rightarrow(0, 1)$ is a $C^1$ map that has minimum $\alpha_{\min }$ and maximum $\alpha_{\max }$, with $\alpha_{\min }<\alpha_{\max }$. In \cite{Gouezel skew product} a partition $\mathcal P_0$ of $\Delta_0=S^1\times(\frac12,1]$ with a certain return time $R$ was given in such a way that $f^{R}$ is an aperiodic induced WGM map (no full branch property) with a coprime block. Moreover, we have
	$	m\{R>n\}\leq {C}{n^{-a}}$
	for some $C>0$, with $a=1/\alpha_{\max}$. Since $\alpha_{\max} <1 $ we have   $ R\in L^{1}(\m).$
	By Theorem~\ref{V.teo:main} there exists a unique ergodic $f$-invariant probability measure  $\mu$ such that $\mu \ll m$.
	Moreover, for given $  \kappa>0$, there is $ 0<\zeta<1 $ such that for all $\varphi \in (R 4, \tau)$ with $ \tau>\frac{2}{\zeta}$ and $ \mathcal{R}_{\infty}(\varphi)< \kappa$, and for all $\psi \in L^\infty(M,\m)$ we have
	\begin{itemize}
		\item if $\tau=\frac{a+1}{\zeta} $, there is some $C{'} > 0$ such that
		$$\Cor_{\mu}(\varphi, \psi \circ f^{n}) \leq C^{'} \left(n^{-a+1} \log n\right);$$
		\item otherwise, there is some $C{'} > 0$ such that
		$$ \Cor_{\mu}(\varphi, \psi \circ f^{n}) \leq C^{'}\max \left\{n^{-a+1}, n^{2-\zeta \tau}\right\}.$$
	\end{itemize}
	By Corollary~\ref{V.CLT Main Cror}, if $\alpha_{max}<1/2$ (thus $a>2$) then CLT is satisfied for all $\varphi \in (R 4, \tau)$, that is not coboundary with some $ \tau>0$ sufficiently large, and by Corollary\ref{V.corol:LD} we also have
	$$		LD_{\mu}(\varphi, \epsilon, n) \leq C^{''}\max \left\{n^{-a+1}, n^{2-\zeta \tau}\right\}
	$$
	for all $\varphi \in (R 4, \tau)$, with $ \tau\neq\frac{a+1}{\zeta}$ and some $C^{''}=C^{''}(\epsilon, \varphi) > 0$.
\end{example}

\section{Decay of Correlations for Tower Maps for larger classes of Observables}\label{V.se:dc on tower}
In this section, we first recall the tower map, and then we state the result of a mixing invariant probability measure for the tower map, coming from \cite{our 1st article}. Furthermore, we design the problem of decay of correlations for this new dynamical system, and devote the rest of this section to address this problem. 
\subsection{Tower Maps}
Assume that $f^{R}: \Delta_{0}\rightarrow \Delta_{0}$ is an induced WGM map for $f: M\rightarrow M $.
We define a \emph{tower}
$$ \Delta=\{(x,\ell): x\in\Delta_{0} \mbox{ and } R(x)> \ell \geq 0 \}   $$

and the tower map  $T: \Delta\rightarrow \Delta$ given by
$$ T(x,\ell)=\left\{
\begin{array}{ll}
	(x,\ell+1), & \mbox{ if } R(x)> \ell+1 \\
	(f^{R}(x),0), & \mbox{ if } R(x)=\ell+1
\end{array}
\right. $$
Note that we can naturally identify the set $\{(x,0)\colon x\in \Delta_0\}\subset\Delta$  with $\Delta_{0}$, and the induced map $T^{R}:\Delta_{0}\rightarrow \Delta_{0}$ with $f^{R}$.
For each $\ell\geq0$, we define the \textit{$\ell$th level} of the tower 
$$\Delta_{\ell}=\{(x,\ell): x\in \Delta_0 \},$$
which is naturally identified with  $\{R>\ell\}\subset\Delta_{0}$. In view of this, we may extend the $\sigma$-algebra $\mathcal A$ and the reference measure $ \m $ on $ \Delta_{0} $ to a $\sigma$-algebra and a measure on $ \Delta $, that we still denote by $\mathcal A$ and $ \m $, respectively. Moreover, the countable partition $ \mathcal{P}_{0} $ of $ \Delta_{0} $ naturally induces an ($\m$ mod 0) partition of each level, that is, if $\mathcal{P}_{0}=\{\Delta_{0,i}\}_{i\in\mathbb N}$ is the partition of $ \Delta_{0} $, then $\{\Delta_{\ell,i}\}_{i \in \mathbb{N}}$, where $\Delta_{\ell,i}=\{(x,\ell)\in \Delta_{\ell}: (x,0)\in\Delta_{0,i}\} $ forms a partition of $\Delta_{\ell}.$
So, the set $ \eta=\{\Delta_{\ell,i}\}_{\ell,i} $  is an ($\m$ mod 0) partition of $ \Delta.$  
For each $ n\geq 1 $, we introduce
$$ 
\eta_{n}=\bigvee_{i=0}^{n-1}T ^{-i}\eta.
$$ 
We can extend the separation time to $ \Delta\times\Delta $ in the following way:
if $ x, y \in \Delta_{\ell} $, then there are unique  $ x_{0}, y_{0} \in \Delta_{0} $ such that $x=T^{\ell}(x_{0}) $ and $ y=T^{\ell}(y_{0}) $, and in this case we set $ s(x, y)=s(x_{0}, y_{0}) $, otherwise set $ s(x, y)=0 $.
It is straightforward to check that $ J_{T} $ is
$$ 
J_{T}(x,\ell)=\left\{
\begin{array}{ll}
	1, & \mbox{ if } R(x)> \ell+1 \\
	J_{f^{R}(x)}, & \mbox{ if } R(x)=\ell+1.
\end{array}
\right.  
$$
Given $0< \beta <1$ we define the following spaces of densities for the tower: 
\begin{equation*}
	\mathcal{F}_{\beta}( \Delta ) = \{\varphi : \Delta\rightarrow \mathbb{R} \colon \exists C_{\varphi} > 0 :|\varphi(x)-\varphi(y)|\leq C_{\varphi}\beta^{s(x,y)}, \forall x,y \in \Delta\}
\end{equation*}
and
\begin{multline*}
	\mathcal{F}_{\beta}^{+}( \Delta ) = \{\varphi\in \mathcal{F}_{\beta}( \Delta ) \colon \exists \hat{C_{\varphi}} > 0 \, \text{ s.t. }\,\\ \varphi(x)>0 \text{ and } \left|\frac{\varphi(x)}{\varphi(y)}-1\right|\leq \hat{C_{\varphi}}\beta^{s(x,y)}, \forall x,y \in\omega \in \eta \}.
\end{multline*}
\begin{theorem}\cite{our 1st article}
	Let  $T: \Delta\rightarrow \Delta$ be the tower map of an aperiodic induced WGM map $ f^R $ with a coprime block and $R\in L^{1}(m)$. Then $ T $ has a unique an exact invariant probability measure  $\nu\ll m$ with $ \frac{d\nu}{dm}\in \mathcal{F}_{\beta}^{+}( \Delta )$, and there is $ C_{0}>0 $  such that $0<\frac{d\nu}{dm}\leq C_{0}.$ Moreover, there exists $C > 0$ such that for all $ x,y \in\omega \in \eta$
	$$\left|\log\frac{\frac{d\nu}{dm}(x)}{\frac{d\nu}{dm}(y)}\right|\leq C\beta^{s(x,y)}.$$
\end{theorem}

Let us define some classes of observables on the tower.
Given a bounded function $ \varphi :\Delta\rightarrow \mathbb{R} $, we define the variation for each $ n\geq 0 $ as :
\begin{equation*}
	v_{n}(\varphi)=\sup\{|\varphi(x)-\varphi(y)|: s(x,y)\geq n\}.
\end{equation*}
Consider the following regularity classes:
 \begin{itemize}
\item  $(V 1, \tau)=\{ \varphi : v_n(\varphi)=\mathcal{O}\left(\tau^n\right)\}$, $ \tau \in(0,1) $.
 
 \item  $ (V 2, \tau)=\{\varphi : v_n(\varphi)=\mathcal{O}\left(\exp \left\{-n^\tau\right\}\right)\}$, $ \tau \in(0,1) $.
 
 \item  $(V 3, \tau)=\{\varphi : v_n(\varphi)=\mathcal{O}\left(\exp \left\{-(\log n)^\tau\right\}\right)\}$, $ \tau>1 $.
 
\item  $ (V 4, \tau)=\{\varphi : v_n(\varphi)=\mathcal{O}\left(n^{-\tau}\right)\}$, $ \tau>1 $.
 
 \end{itemize}
 Now we define a measurable semi-conjugacy  $\pi : \Delta \rightarrow M$ between the tower map $T$ and the map $f$, by $\pi(x,\ell)=f^{\ell}(x)$. We have  $\pi\circ T = f\circ \pi$ and $ \pi_{*} \nu$ coincides with the $f$-invariant measure $\mu$ given by Theorem \ref{V.teo:main}. It is an immediate consequence that for  all $\varphi, \psi : M\rightarrow \mathbb{R}$,
\begin{equation}\label{conjugacy}
	\Cor_{\mu}(\varphi, \psi \circ f^{n}) = \Cor_{\nu}(\varphi \circ \pi, \psi \circ \pi \circ T^{n}).
\end{equation}
Given a regularity for $ \varphi :M\rightarrow \mathbb{R} $  in terms of $ 	R_{\epsilon}(\varphi) $, we need the regularity of $ \varphi\circ\pi:\Delta\rightarrow \mathbb{R} $.
Particularly, we want that the observable classes $(R1-R4)$ on $ M $ correspond to the classes $(V1-V4)$ on $\Delta$. For that we have the following lemma.

\begin{lemma}\cite{Lynch paper}\label{V.lemma: phi circ pi}
	Let $T: \Delta \rightarrow \Delta$ be the tower map of an induced expanding WGM map $f^{R} :\Delta_{0} \rightarrow \Delta_{0}$, then
		\begin{itemize}
		\item $\varphi \in(R i, \tau)$ $ \Rightarrow $ $\varphi \circ \pi \in(V i, \tau^{\prime} )$, with $\tau^{\prime}<\tau$, for $i=1,2,3$;
%
		\item $\varphi \in(R 4, \tau)$ $ \Rightarrow $ $\varphi \circ \pi \in(V 4, \tau)$.
	\end{itemize}
\end{lemma}

Consider the following spaces which are important for the coupling argument:
\begin{equation*}
	\mathcal{I}( \Delta ) = \{\varphi : \Delta\rightarrow \mathbb{R} : v_{n}(\varphi)\rightarrow 0 \},
\end{equation*}
and 
\begin{multline*}
	\mathcal{I}^{+}( \Delta ) = \Bigl\{\varphi\in \mathcal{I}( \Delta ) : \exists C_{\varphi}^{\prime} > 0 \text{ s.t. } \varphi(x)>0 \text{ and }\\ \left|\frac{\varphi(x)}{\varphi(y)}-1\right|\leq C_{\varphi}^{\prime}v_{s(x,y)}(\varphi)+C^{\prime\prime}\beta^{s(x,y)}, \forall x,y \in\omega \in \eta \Bigr\},
\end{multline*}
where $C^{\prime\prime}>0$ is a fixed constant to be specified in Lemma \ref{V.lemma:phi*}.
Given $\varphi\in L^\infty(\Delta)$ we define
$$
\varphi^{\ast}=\frac{1}{\int(\varphi+2||\varphi||_{\infty}+1)d\nu}(\varphi+2||\varphi||_{\infty}+1).
$$
\begin{lemma}\cite{Alves Book}\label{V.D.Correlation concection of T and convergence of equlibrium}
	For all $\varphi\in L^{\infty}(m)$ with $ \varphi\neq 0 $ we have
	\begin{itemize}
		\item[i)] $\frac{1}{3}\leq \varphi^{\ast}\leq 3$;
		\item[ii)]  $\Cor_{\nu}(\varphi,\psi \circ T^{n})\leq3(||\varphi||_{\infty}+1)||\psi||_{\infty}|T_{\ast}^{n}\lambda-\nu|$, for all $\psi\in L^{\infty}(m)$, where $\lambda$ is the probability measure on $\Delta$ such that  $\frac{d\lambda}{dm}=\varphi^{\ast}\frac{d\nu}{dm}.$
	\end{itemize}
\end{lemma}
From $ii)$ of Lemma \ref{V.D.Correlation concection of T and convergence of equlibrium} to obtain decay of correlation for the tower map it is enough to estimate $ |T_{*}^n\lambda-\nu| $. 
For that we have the following.

\begin{proposition}\label{V.Convergence to equilibrium}
	Let $ T:\Delta\rightarrow \Delta $ be the tower map of an aperiodic induced WGM map $ f^{R}$ with a coprime block and $ R\in L^{1}(m) $. If $ \nu $ is the unique mixing $ T$-invariant probability measure such that $ \frac{d\nu}{dm}\in \mathcal{I}^{+}( \Delta ) $, then
	\begin{enumerate}
		\item if $m\{ R >n \}=\mathcal{O} (n^{-a})$ for some $a>1$, given $ \kappa>0$, there is $0<\zeta<1$ such that, for any probability measure  $ \lambda $ with $ \frac{d\lambda}{dm}\in \mathcal{I}^{+}( \Delta ) $, $ v_{n}(\frac{d\lambda}{dm})=\mathcal{O}( n^{-\tau})$, for some $\tau>\frac{2}{\zeta}$, and $v_0(\frac{d\lambda}{dm})< \kappa$ we have
		\begin{itemize}
			\item[i)] if $\tau=\frac{a+1}{\zeta}$, then $|T_{*}^{n}\lambda-\nu|=\mathcal{O}\left(n^{1-a} \log n\right)$;
			\item[ii)]  otherwise, $|T_{*}^{n}\lambda-\nu|=\mathcal{O}\left(\max \left(n^{1-a}, n^{2-\zeta \tau}\right)\right)$.
		\end{itemize}
		\item if $m\{ R >n \}=\mathcal{O} (e^{-cn})$ for some $c>0$, then 
		\begin{itemize}
			\item[i)]  for any probability measure  $ \lambda $ with $ \frac{d\lambda}{dm}\in \mathcal{I}^{+}( \Delta ) $, and $ v_{n}(\frac{d\lambda}{dm})=\mathcal{O}(\tau^n)$ for some $\tau\in (0,1)$, we have
			  $$|T_{*}^{n}\lambda-\nu|=\mathcal{O}(e^{-c^{\prime}n}) \text{ for some } c^{\prime}>0;$$
			\item[ii)]  for any probability measure  $ \lambda $ with $ \frac{d\lambda}{dm}\in \mathcal{I}^{+}( \Delta ) $, and $ v_{n}(\frac{d\lambda}{dm})=\mathcal{O}(e^{-n^\tau})$ for some $\tau\in (0,1)$, we have
			 $$|T_{*}^{n}\lambda-\nu|=\mathcal{O}(e^{-n^{\tau^{\prime}}}) \text{ for every } \tau^{\prime}<\tau;$$
			\item[iii)]  for any probability measure  $ \lambda $ with $ \frac{d\lambda}{dm}\in \mathcal{I}^{+}( \Delta ) $, and $ v_{n}(\frac{d\lambda}{dm})=\mathcal{O}(e^{-(\log(n))^\tau})$ for some $\tau>1$, we have
		  $$|T_{*}^{n}\lambda-\nu|=\mathcal{O}(e^{-(\log n)^{\tau^{\prime}}}) \text{ for every } \tau^{\prime}<\tau;$$
			\item[iv)]  given $ \kappa>0$, there is $\zeta<1$ such that, for any probability measure  $ \lambda $ with $ \frac{d\lambda}{dm}\in \mathcal{I}^{+}( \Delta ) $, $ v_{n}(\frac{d\lambda}{dm})=\mathcal{O}( n^{-\tau})$, for some $\tau>\frac{1}{\zeta}$, and $v_0(\frac{d\lambda}{dm})< \kappa$ we have
		 $$|T_{*}^{n}\lambda-\nu|=\mathcal{O}(n^{1-\zeta \tau}).$$
		\end{itemize}
	\end{enumerate}
\end{proposition}
 \begin{remark}\label{main thm.p}
We observe that from \eqref{conjugacy}, to obtain  decay of correlations for $(f, \mu)$ it therefore suffices to obtain a decay of correlations for $ (T, \nu) $. This means that by using Lemma~\ref{V.lemma: phi circ pi}, Lemma~ \ref{V.D.Correlation concection of T and convergence of equlibrium} and Proposition~\ref{V.Convergence to equilibrium}, we can get the proof of Theorem \ref{V.teo:main}. So it remains to prove Proposition~\ref{V.Convergence to equilibrium}.
\end{remark}
 As we want to use Lemma \ref{V.D.Correlation concection of T and convergence of equlibrium} and Proposition \ref{V.Convergence to equilibrium} in order to get proof of Theorem \ref{V.teo:main}, we need to check the necessary regularity of 
 $\frac{d\lambda}{dm}=\varphi^{\ast}\frac{d\nu}{dm}$, for a given $\varphi$ in classes $ (V1-V4)$.
 
 \begin{lemma}\label{V.lemma:phi*}
	If $\varphi \in(V j, \tau)$, and $ \tau \in(0, 1) $  for  $ j=1,2 $ or $ \tau>1 $  for  $ j=3,4 $ then
	$\varphi^{\ast}\frac{d\nu}{dm}\in \mathcal{I}^{+}( \Delta )$.
\end{lemma}
 \begin{proof}
Set $\rho=\frac{d\nu}{dm}$.
 Since $0<\rho \leq C_{0}$ and from Lemma \ref{V.D.Correlation concection of T and convergence of equlibrium} we have
\begin{equation*}
	\begin{aligned}
		|\varphi^{\ast}(x)\rho(x)-\varphi^{\ast}(y)\rho(y)|&\leq|\varphi^{\ast}(x)(\rho(x)-\rho(y))|+|\rho(y)(\varphi^{\ast}(x)-\varphi^{\ast}(y))|\\
		&\leq 3|\rho(x)-\rho(y)|+C_{0}|\varphi(x)-\varphi(y)|.
	\end{aligned}
\end{equation*} 
 Thus $ v_n(\varphi^{\ast}\rho)\leq 3C_{\rho}\beta^{n}+C_{0}v_{n}(\varphi).$ 
This implies that $\varphi^{\ast}\rho\in \mathcal{I}( \Delta )$.
From Lemma \ref{V.D.Correlation concection of T and convergence of equlibrium}, we have
\begin{equation}\label{claim eq}
	\begin{aligned}
		\left|\frac{\varphi^{\ast}(x)}{\varphi^{\ast}(y)}-1\right|&=\frac{1}{\varphi^{\ast}(y)}\left|\varphi^{\ast}(x)-\varphi^{\ast}(y)\right|\\
		&\leq 3|\varphi(x)-\varphi(y)|\\
		&\leq3 v_{s(x,y)}(\varphi).
	\end{aligned}
\end{equation}

Since $\frac{1}{9}\leq \frac{\varphi^{\ast}(x)}{\varphi^{\ast}(y)}\leq 9 $, for all $ x, y\in \Delta $, then there exists $ K_{1}^{\prime}>0 $ such that 
\begin{equation}\label{V.L eq for log-densisity}
	\left|\log\frac{\varphi^{\ast}(x)}{\varphi^{\ast}(y)} \right|\leq K_{1}^{\prime}\left|\frac{\varphi^{\ast}(x)}{\varphi^{\ast}(y)}-1\right|.
\end{equation}
It follows from \eqref{claim eq} and \eqref{V.L eq for log-densisity} that for all $ x,y\in \omega\in \eta $,
\begin{equation}\label{V.L final log density}
	\begin{aligned}
		\left|\log\frac{\varphi^{\ast}(x)\rho(x)}{\varphi^{\ast}(y)\rho(y)} \right|&\leq\left|\log\frac{\varphi^{\ast}(x)}{\varphi^{\ast}(y)} \right|+\left|\log\frac{\rho(x)}{\rho(y)} \right|\\
		&\leq 3K_{1}^{\prime} v_{s(x,y)}(\varphi)+C\beta^{s(x,y)}.
	\end{aligned}
\end{equation}

Since $ 0<\frac{\varphi^{\ast}(x) \rho(x)}{\varphi^{\ast}(y) \rho(y)}\leq 9e^{C} $, for all $ x, y\in \omega\in \eta $, then we also have some uniform constant $ K_{2}^{\prime}>0 $ that for all $ x, y\in \omega\in \eta $,
\begin{equation*}
	\left|\frac{\varphi^{\ast}(x)\rho(x)}{\varphi^{\ast}(y)\rho(y)}-1 \right|\leq K_{2}^{\prime}\left|\log\frac{\varphi^{\ast}(x)\rho(x)}{\varphi^{\ast}(y)\rho(y)} \right|\leq K_{2}^{\prime}(3K_{1}^{\prime} v_{s(x,y)}(\varphi)+C\beta^{s(x,y)}).
\end{equation*}
Hence $\varphi^{\ast}\rho\in \mathcal{I}^{+}( \Delta )$ with $ C^{\prime\prime}=K_{2}^{\prime}C$.
\end{proof}
 We would like to point out a remark about $ \varphi^{\ast}\frac{d\nu}{dm} $, which is important because we will assume this king of regularity in the coupling argument.
\begin{remark}
	From the previous proof we can see that there exists $C_{\varphi^*\rho}''>0$ such that for all  $ x,y\in \omega\in \eta $,
	\begin{equation*}
		\left|\log\frac{\varphi^{\ast}(x)\rho(x)}{\varphi^{\ast}(y)\rho(y)} \right|\leq C_{\varphi^*\rho}''v_{s(x,y)}(\varphi)+C\beta^{s(x,y)}.
	\end{equation*}
\end{remark}
\subsection{Coupling}
 We give the proof of Proposition \ref{V.Convergence to equilibrium} in these remaining sections.

Let $ \lambda_{1} $ and $ \lambda_{2} $ be probability measures on $ \Delta $ with $\varphi_{1} =\frac{d\lambda_{1}}{dm}, \varphi_{2} =\frac{d\lambda_{2}}{dm}\in \mathcal{I}^{+}( \Delta )$. 
Let $P= \lambda_{1}\times\lambda_{2} $ be the product measure on $ \Delta\times\Delta .$ 
We consider the product transformation $ T\times T :\Delta\times\Delta\rightarrow \Delta\times\Delta $, and let $ \pi_{1}, \pi_{2} :\Delta\times\Delta\rightarrow \Delta $ be the projections onto the first and second coordinates, respectively. 
We use $ \eta\times\eta $ to denote the product partition of $ \Delta\times\Delta $ and $ (\eta\times\eta)_{n}=\bigvee_{i=0}^{n-1}(T\times T) ^{-i}(\eta\times\eta). $
Notice that 
\begin{equation}\label{E32}
	T^{n}\circ\pi_{1}=\pi_{1}\circ(T\times T) ^{n} \text{ and } T^{n}\circ\pi_{2} = \pi_{2}\circ(T\times T) ^{n}.
\end{equation}
Let $ \hat{R}: \Delta\rightarrow \mathbb{N} $ be the return time to $ \Delta_{0} $, defined for $ x\in \Delta $ by
$$ \hat{R}(x)=\min\{n\geq0 : T^{n}(x)\in \Delta_{0} \}.$$

We fix some integer $ n_{0} >0 $ such that  for any $ \omega\in \mathcal{P}_{0} $,
\begin{equation*}
	m(T^{-n}(\Delta_{0})\cap f^{R}(\omega))\geq \text{ some } \gamma_{0}>0 \text{ for }  n\geq  n_{0}. 
\end{equation*}
The above choice of $ n_{0} $ is important for Lemma \ref{V.L27}. 

Let us now introduce a sequence of \emph{stopping times} $ 0=\tau_{0}< \tau_{1}< \tau_{2}... $ on $ \Delta\times\Delta $,
\begin{equation*}
	\begin{array}{lcl}
		\tau_{1}&=&n_{0}+\hat{R}\circ T^{ n_{0}}\circ\pi_{1}\\
		\tau_{2}&=&n_{0}+\tau_{1}+\hat{R}\circ T^{ n_{0}+\tau_{1}}\circ\pi_{2}\\ 
		\tau_{3}&=&n_{0}+\tau_{2}+\hat{R}\circ T^{ n_{0}+\tau_{2}}\circ\pi_{1}\\
		\tau_{4}&=&n_{0}+\tau_{3}+\hat{R}\circ T^{ n_{0}+\tau_{3}}\circ\pi_{2}\\
		&\vdots&
	\end{array}
\end{equation*}
We define the simultaneous return to $ \Delta_{0} $ $ S:\Delta\times\Delta\rightarrow \mathbb{N} $ by
\begin{equation}\label{E30}
	S(x, y)=\min_{i\geq 2}\{\tau_{i}(x,y): (T^{\tau_{i}(x, y), }(x),T^{\tau_{i}(x, y), }(y))\in \Delta_{0}\times\Delta_{0}\},
\end{equation}
which is well defined $ m\times m $ almost everywhere.
Let $\xi_{0}<\xi_{1}<\xi_{2}<\xi_{3}...$ be an increasing sequence of partitions on $ \Delta\times\Delta $ defined as follows. As usual, given a partition $ \xi $, we denote $\xi(x)$  the element
of $ \xi $ containing $ x $.
First we take $\xi_{0}= \eta\times\eta$.
Now we describe the general inductive step in the construction of partitions $ \xi_{k} $. Assume that $ \xi_{j} $ has been constructed for all $ j<k $. The definition of $ \xi_{k} $ depends on whether $ k $ is odd or even. For definiteness we assume that $ k $ is odd. The construction for $ k $ even is the same apart from the change in the role of the first and second components.   
We let $ \xi_{k}=\{\xi_{k}(\overline{x}): \overline{x} \in \Delta\times\Delta \}$, where
\begin{equation*}
	\xi_{k}(\overline{x})= \bigvee_{i=0}^{\tau_{k}(\overline{x})-1}(T^{-i}(\eta))(x)\times\pi_{2}(\xi_{k-1}(\overline{x})).
\end{equation*}
The following lemmas, which are similar to \cite[Lemmas 5.10, 5.11]{our 1st article}, are crucial to estimate $P(S>n)$. The only difference here is the dependence of the constants $\epsilon_0$ and $C_2$.  
\begin{lemma}\label{V.L27}
	There exists $ \epsilon_{0}>0$, depending on $ C_{\varphi_1}', C_{\varphi_2}' $, $ v_{0}(\varphi_{1}) $, $ v_{0}(\varphi_{2}) $ such that for all $ k\geq 1 $ and $ \tau \in\xi_{k}  $ with $ S|_\tau>\tau_{k-1}$, we have 
	\begin{equation*}
		P(S=\tau_{k} | \tau)\geq \epsilon_{0}.
	\end{equation*}
	Moreover, the dependence of $\epsilon_{0}$ on $ C_{\varphi_1}', C_{\varphi_2}' $, $ v_{0}(\varphi_{1}) $, $ v_{0}(\varphi_{2}) $ can be removed if we take $ k $ sufficiently large. 
	
\end{lemma}
\begin{lemma}\label{V.L28}
	There exists $ D_{1}$, depending on $ C_{\varphi_{1}}', C_{\varphi_{2}}' $, $ v_{0}(\varphi_{1}) $, $ v_{0}(\varphi_{2}) $ such that for all $n, k\geq 0 $ and for all $ \tau \in\xi_{k} $, we have 
	\begin{equation*}
		P(\tau_{k+1}- \tau_{k}>n+n_{0} | \tau)\leq D_{1}m\{\hat{R}>n\}.
	\end{equation*}
	Moreover, the dependence of $D_{1}$ on  $ C_{\varphi_{1}}', C_{\varphi_{2}}' $, $ v_{0}(\varphi_{1}) $, $ v_{0}(\varphi_{2}) $ can be removed if we take $ k $ sufficiently large.
	
\end{lemma}
Finally we can obtain the estimate for $ P\{S>n\}$. 
Note that 
\begin{equation*}
	m\{\hat{R}>n\}=\sum_{\ell >n}m(\Delta_{\ell})=\sum_{\ell >n}m\{R>n\},
\end{equation*}
which together Lemma~\ref{V.L27}, Lemma~\ref{V.L28}, and \cite[Proposition 3.46, Proposition 3.48]{Alves Book}, provide the following result.
\begin{lemma}\, 
\begin{enumerate}\label{estimate for simulatainious retun time}
	\item If $m\{R>n\}\leq Cn^{-a}$, for some $C>0$ and $a>1$, then $P\{S>n\}\leq~ C^{\prime}n^{-a+1}$ for some $ C^{\prime}>0 $
	\item If $m\{R>n\} \leq C e^{-c n}$, for some $C, c>0$, then 
	$ P\{S>n\}\leq C^{\prime} e^{-c_{0} n}$ for some  $C^{\prime}, c_{0}>0.$
\end{enumerate}
\end{lemma}
We are now going to estimate  $|T_{*}^{n}\lambda_{1}-T_{*}^{n}\lambda_{2}|$. 
Consider the induced map $ \widetilde{T}=(T\times T)^{S}: \Delta\times\Delta\rightarrow \Delta\times\Delta ,$ with $ S $ as in \eqref{E30}, and the functions $ 0=S_{0}<S_{1}<S_{2}<..., $ defined for each $ n\geq 1 $ as $ S_{n}=S_{n-1}+S\circ(T\times T)^{S_{n-1}}. $
Note that $ \widetilde{T}^{n}=(T\times T)^{S_{n}}. $

Let $ \widetilde{\xi} $ be the partition of $ \Delta\times\Delta $ into the rectangles $ \Omega $ on which $S$ is constant and $ \widetilde{T} $ maps $ \Omega $ bijectively onto a union of elements of $  \eta\times\eta|_{\Delta_0 \times \Delta_0} .$ 
Without loss of generality, we assume that for any $ \Omega\in \widetilde{\xi}|_{\Delta_0 \times \Delta_0}$, there exists $ \omega_{j}\times\omega_{j^{\prime}}\in  \eta\times\eta|_{\Delta_0 \times \Delta_0}$ such that $ \Omega\subset \omega_{j}\times\omega_{j^{\prime}} .$
For each $ n\geq 1, $ define 
$$
\tilde{\xi}_n=\bigvee_{j=0}^{n-1} \widetilde{T}^{-j} (\tilde{\xi}).
$$
Each $\tilde{\xi}_n$ is a partition into sets $\Omega \subset \Delta \times \Delta$ on which $S_n$ is constant and $\widetilde{T}^n$ maps $\Omega$ bijectively onto an $\m\times\m$ mod 0 union of elements of  $ \eta\times\eta|_{\Delta_0 \times \Delta_0}$. Let us introduce a separation time in $\Delta \times \Delta$, defining for each $u, v \in \Delta \times \Delta$
$$\tilde{s}(u, v)=\min \left\{n \geq 0: \widetilde{T}^n(u)\right.\text{ and }\widetilde{T}^n(v) \text{ lie in distinct elements of }\left.\tilde{\xi}\right\}.$$

Let $ \Phi(x,y)=\varphi_{1}(x)\varphi_{2}(y)$ and let $ C_{\varphi_{1}}'' $ and $ C_{\varphi_{2}}''$ be constants such that, for $ i=1, 2 $
\begin{equation*}
	\left|\log \frac{\varphi_{i}(x)}{\varphi_{i}(y)} \right|\leq C_{\varphi_{i}}'' v_{s\left(x,y\right)}(\varphi_{i}) \text{ for all } x,y \in \omega\in \eta.
\end{equation*}
Set, $ C_{\Phi}= C_{\varphi_{1}}''+C_{\varphi_{2}}'' $ and $ v_{n}(\Phi)=\max\{v_{n}(\varphi_{1}), v_{n}(\varphi_{2})\}. $
 If we consider, $ \left|\log \frac{\varphi_{1}(x)}{\varphi_{1}(y)} \right|\leq  C_{\varphi_{1}}'' v_{s\left(x,y\right)}(\varphi_{1})+C_{\varphi_{2}}^{\prime\prime} v_{s\left(x,y\right)}(\varphi_{2})$, then we set, $ C_{\Phi}=  C_{\varphi_{1}}''+2C_{\varphi_{2}}^{\prime\prime}$. Similar, to \cite[Proposition 5.13]{our 1st article} we get the following.
\begin{proposition}\label{V.P38}
	There exists $D_{2}>0$ depending on $C_{\Phi}, v_{0}(\Phi)$ such that for all $n, k \geq 0$, we have
	$$
	P\left\{S_{k+1}-S_k>n\right\} \leq D_{2}(m \times m)\{S>n\}.
	$$
\end{proposition}
The following lemma play an important role in this setting. We remark that \eqref{V.4} and \eqref{V.5} in Lemma \ref{ Chosing a sequence of epsilon(i)} are important for Lemma \ref{V.L34} (see \cite{Lynch paper}), and \eqref{V.6} is useful for Sections \ref{Polynomial Return time} and \ref{Exponentional Return time}.

\begin{lemma}\cite{Lynch paper}\label{ Chosing a sequence of epsilon(i)}
	Given a sequence $v_i(\Phi)$, there exists a sequence $\epsilon_i^{\prime} \leq \frac{1}{2}$ such that
	\begin{equation}\label{V.4}
		v_i(\Phi) \prod_{j=1}^i\left(1+\epsilon_j^{\prime}\right) \leq D_{3},
	\end{equation}
	\begin{equation}\label{V.5}
		\sum_{j=1}^i\left(\prod_{k=j}^i\left(1+\epsilon_k^{\prime}\right)\right) \beta^{i-j+1} \leq D_{3}
	\end{equation}
	for some sufficiently large constant $ D_{3} $ depending on $ v_{0}(\Phi).$
		Moreover, for any $\bar{D}>1$ and $ \bar{\delta}>0 $ with $\epsilon_i=\bar{\delta} \epsilon_i^{\prime}$,
	\begin{equation}\label{V.6}
		\prod_{j=1}^i\left(1-\frac{\epsilon_j}{\bar{D}}\right) \leq \tilde{C} \max \left(v_i(\Phi)^{\frac{\bar{\delta}}{\bar{D}}}, \theta^i\right)
	\end{equation}
	for some $\theta<1$ depending only on $T$, and some $\tilde{C}>0$.
\end{lemma}
We are going to define a sequence of densities $\tilde{\Phi}_0 \geq \tilde{\Phi}_1 \geq \tilde{\Phi}_2 \geq \cdots$ in $ \Delta\times\Delta $, for the total measure remaining in the system after $ n $ iteration bt $ \widetilde{T}.$
Consider constant $ \bar{\delta}>0$, depending on $ C_{\Phi}$ and $ v_{0}(\Phi)$ to be chosen properly such that defining
\begin{equation}\label{V.E49}
	\tilde{\Phi}_i(u)= \begin{cases}\Phi(u), & \text { if } i=0 \\ \tilde{\Phi}_{i-1}(u)-\epsilon_i J_{\widetilde{T}^i}(u) \min _{v \in \Omega_i(u)} \frac{\tilde{\Phi}_{i-1}(v)}{J_{\tilde{T}^i}(v)}, & \text { if } i>0.\end{cases}
\end{equation}
for all $i\geq1$ with $\epsilon_{i}=\bar{\delta}\epsilon_{i}^{\prime}>0$, the following holds (for details see~\cite[Lemma 2]{Lynch paper}):
\begin{lemma}\label{V.L34}
	There exists $ D_{4}>1 $, depending on $ C_{\Phi}$ and $ v_{0}(\Phi)$, such that for all $ i\geq1 $
	we have 
	$$\tilde{\Phi}_i\leq (1-\frac{\varepsilon_{i}}{D_{4}})\tilde{\Phi}_{i-1} \text{ on } \Delta\times\Delta.$$
\end{lemma}

Now we are going to define the corresponding densities in real time iterations under $ T\times T $.
Let us introduce functions  $\Phi_0 \geq \Phi_1 \geq \Phi_2 \geq \cdots$ on $ \Delta\times \Delta $ such that for $ v\in \Delta\times \Delta $, we define
\begin{equation*}
	\Phi_n(v)=\tilde{\Phi}_i(v), \quad \text { if } \quad S_i(v) \leq n<S_{i+1}(v) .
\end{equation*}
For all $ n\geq 1 $ we have
\begin{equation}\label{E64}
	\Phi=\Phi_n+\sum_{k=1}^n\left(\Phi_{k-1}-\Phi_k\right).
\end{equation}
For each $ k\geq 1 $, let $A_k=\cup_{i} A_{k, i}$,
where $A_{k, i}=\{u \in \Delta \times \Delta: k=$ $\left.S_i(u)\right\}$. Note that $A_{k, i} \cap A_{k, j}=\emptyset$ for $i \neq j$ (because $ S_i(u)\neq S_j(u) $ for $i \neq j$), and each $A_{k, i}$ is a union of elements of $\tilde{\xi}_i$.
\begin{remark}\label{R3}
	By definition, for any $\Omega \in \tilde{\xi}_i|_{A_{k, i}}$, we have $ S_{i-1}|_\Omega< S_{i}|_\Omega=k $, and $S_{i-1}|_\Omega\leq k-1  $. This implies that $\Phi_{k-1}-\Phi_k=\tilde{\Phi}_{i-1}-\tilde{\Phi}_i$ on $\Omega \in \tilde{\xi}_i|_{A_{k, i}}$, and $\Phi_k=\Phi_{k-1}$ on $\Delta \times \Delta \backslash A_k$.
\end{remark}
We have the following main result in this subsection which is important step to prove Proposition \ref{V.Convergence to equilibrium}.
\begin{proposition}\label{V.first Matching Proposition}
	There exists $D_{5}>0$, depending on $C_{\Phi}, v_{0}(\Phi)$, such that for all $ n\geq 1$,
	\begin{equation*}
	\left|T_*^n \lambda_{1}-T_*^n \lambda_{2}\right|\leq 2 P\{S>n\}+ D_{5}\sum_{i=1}^{n}\prod_{j=1}^i\left(1-\frac{\varepsilon_j}{D_{4}}\right) (i+1)(m\times m)\left\{S>\frac{n}{i+1}\right\},
\end{equation*}
\end{proposition}

\begin{proof}
	From~\eqref{E32} and~\eqref{E64}, for each $ n\geq 1$  we have 
	\begin{equation}\label{step1}
		\left|T_*^n \lambda_{1}-T_*^n \lambda_{2}\right| \leq I_1+ I_2,
	\end{equation}
	where
	\begin{equation*}
		I_1= \left|\left({\pi_1}_*-{\pi_2}_*\right)(T \times T)_*^n\left(\Phi_n(m \times m)\right)\right|
	\end{equation*}
	and
	\begin{equation*}
		I_2=\sum_{k=1}^n\left|\left({\pi_1}_*-{\pi_2}_*\right)\left[(T \times T)_*^n\left(\left(\Phi_{k-1}-\Phi_k\right)(m \times m)\right)\right]\right|.
	\end{equation*}
In one hand, from~\cite[Lemma 3]{Lynch paper} we get
	\begin{equation}\label{step2}
		I_1\leq 2 P\{S>n\}+2 \sum_{i=1}^{\infty}\prod_{j=1}^i\left(1-\frac{\epsilon_j}{D_{4}}\right) P\left\{S_i \leq n<S_{i+1}\right\}.
	\end{equation}
	On the other hand,
	\begin{equation}\label{DE72}
		\begin{aligned}
			I_2&\leq 2\sum_{k=1}^n \int_{\Delta\times\Delta} \left(\Phi_{k-1}-\Phi_k\right) d(m \times m)\\
			&= 2\sum_{k=1}^n \int_{\Delta\times\Delta\setminus A_k} \left(\Phi_{k-1}-\Phi_k\right) d(m \times m)+2\sum_{k=1}^n \int_{A_k} (\left(\Phi_{k-1}-\Phi_k\right) d(m \times m).
		\end{aligned}
	\end{equation}
	From Remark \ref{R3}, we have 
	\begin{equation}\label{DE73}
		2\sum_{k=1}^n \int_{\Delta\times\Delta\setminus A_k} \left(\Phi_{k-1}-\Phi_k\right) d(m \times m)=0,
	\end{equation}
	and writing $A_k=\cup_{i=1}^{\infty} A_{k, i}$, we have
	\begin{equation}\label{DE75}
			\int_{A_k} \left(\tilde{\Phi}_{k-1}-\tilde{\Phi}_k\right) d(m \times m)=\sum_{i=1}^{\infty} \int_{A_{k,i}} \left(\tilde{\Phi}_{i-1}-\tilde{\Phi}_i\right) d(m \times m).
	\end{equation}
	
	 By~\eqref{V.E49} and  Lemma~\ref{V.L34} we have
	\begin{equation}\label{DE76}
		\begin{aligned}
			\sum_{i=1}^{\infty} \int_{A_{k,i}} \left(\tilde{\Phi}_{i-1}-\tilde{\Phi}_i\right) d(m \times m)
			&=\sum_{i=1}^{\infty} \int_{A_{k,i}} \epsilon_{i} J_{\widetilde{T}^i}(u) \min _{v \in A_{k,i}} \frac{\tilde{\Phi}_{i-1}(v)}{J_{\tilde{T}^i}(v)} d(m \times m)\\
			&\leq \sum_{i=1}^{\infty} \int_{A_{k,i}}\epsilon_{i}   \tilde{\Phi}_{i-1}(u) d(m \times m)\\
			&\leq \sum_{i=1}^{\infty} \int_{A_{k,i}} \epsilon_{i}\prod_{j=1}^{i-1}\left(1-\frac{\varepsilon_j}{D_{4}}\right) \Phi(u) d(m \times m)\\
			&=\sum_{i=1}^{\infty}\epsilon_{i} \prod_{j=1}^{i-1}\left(1-\frac{\varepsilon_j}{D_{4}}\right)P\{S_i=k\}.\\
		\end{aligned}
	\end{equation}
	It follows from (\ref{DE72}), (\ref{DE73}), (\ref{DE75}) and  (\ref{DE76}) that
	\begin{equation}\label{DE77}
		I_2\leq2 \sum_{i=1}^{\infty} \epsilon_{i} \prod_{j=1}^{i-1}\left(1-\frac{\epsilon_j}{D_{4}}\right)P\left(\bigcup_{k=1}^{n}\{S_i=k\}\right).
	\end{equation}
	 For each $ 1\leq k\leq n $, we have $ \{S_i=k\}\subset \{S_i\leq n< S_{i+1}\} $, and this implies that $ \cup_{k=1}^{n}\{S_i=k\}\subset\{S_i\leq n< S_{i+1}\} $, from~\eqref{DE77} we have
	\begin{equation}\label{step3}
		I_2\leq \frac{2\bar{\delta}D_{4}}{2D_{4}-\bar{\delta}} \sum_{i=1}^{\infty} \prod_{j=1}^{i}\left(1-\frac{\epsilon_j}{D_{4}}\right)P\{S_i\leq n< S_{i+1}\}.
	\end{equation}
	It follows from \eqref{step1}, \eqref{step2} and \eqref{step3},
	\begin{equation}\label{cor.step}
		\begin{aligned}
			\left|T_*^n \lambda_{1}-T_*^n \lambda_{2}\right|&\leq 
			2P\{S>n\}+2\left(1+\frac{\bar{\delta}D_{4}}{2D_{4}-\bar{\delta}}\right)\sum_{i=1}^{\infty}\prod_{j=1}^{i}\left(1-\frac{\epsilon_j}{D_{4}}\right) P\left\{S_i \leq n<S_{i+1}\right\}.
		\end{aligned}
	\end{equation}
For each $ i\geq 1 $ we have 
\begin{equation}\label{DDE79}
	P\left\{S_i \leq n<S_{i+1}\right\} \leq \sum_{j=0}^i P\left\{S_{j+1}-S_j>\frac{n}{i+1}\right\},
\end{equation}
and Proposition \ref{V.P38},  
\begin{equation}\label{DE78}
	P\left\{S_{j+1}-S_j>\frac{n}{i+1}\right\}\leq D_{2}(m\times m)\left\{S>\frac{n}{i+1}\right \}.
\end{equation}
Combining (\ref{DDE79}) and (\ref{DE78}), we get
\begin{equation*}
	P\left\{S_i \leq n<S_{i+1}\right\} \leq D_{2} (i+1)(m\times m)\left\{S>\frac{n}{i+1}\right\},
\end{equation*}
which, together with \eqref{cor.step}, implies
	\begin{equation*}
	\left|T_*^n \lambda_{1}-T_*^n \lambda_{2}\right|\leq 2 P\{S>n\}+ D_{5}\sum_{i=1}^{n}\prod_{j=1}^i\left(1-\frac{\varepsilon_j}{D_{4}}\right) (i+1)(m\times m)\left\{S>\frac{n}{i+1}\right\},
\end{equation*}
with $ D_{5}=2D_{2}\left(1+\frac{\bar{\delta}D_{4}}{2D_{4}-\bar{\delta}}\right).$ We recall that $\bar{\delta}$, $D_{2}$ and $ D_{4} $ (hence $D_5$) depend on $ C_{\Phi}, v_{0}(\Phi) $.

\end{proof}
Set $\zeta=\frac{\bar{\delta}}{D_{4}}$, which can be seen to depend only on $ C_{\Phi}$ and $ v_{0}(\Phi)$. In Subsections \ref{Polynomial Return time} and \ref{Exponentional Return time}, we let $D$ generic constant, allowed to depend only on $T$ and $\Phi$. 
\subsection{Polynomial Return time}\label{Polynomial Return time}
If $m\{R>$ $n\}=\mathcal{O}( n^{-a} )$, for some $ a>1 $, then by Lemma \ref{estimate for simulatainious retun time},
\begin{equation}\label{V.Final estimatr of simulatainious RT}
	P\{S>n\}=\mathcal{O}(n^{-a+1}).
\end{equation} 
By similar arguments used to estimate $P\{S>n\}$, we have 
\begin{equation}\label{V.estimatr of simulatainious RT with mxm}
	(m\times m)\left\{S>\frac{n}{i+1}\right\}=\mathcal{O}\left(\left(\frac{n}{i+1}\right)^{-a+1}\right).
\end{equation}

Class $(V4, \tau)$: Assume that $v_n(\Phi)=\mathcal{O}\left(n^{-\tau}\right)$, for some $\tau>\frac{2}{\zeta}$. By Lemma \ref{ Chosing a sequence of epsilon(i)} we have
\begin{equation}\label{V4 for Polynomial}
	\prod_{j=1}^i\left(1-\frac{\varepsilon_j}{D_{4}}\right)=\mathcal{O} (i^{-\zeta \tau}).
\end{equation}
By using \eqref{V.Final estimatr of simulatainious RT}, \eqref{V.estimatr of simulatainious RT with mxm} and \eqref{V4 for Polynomial} in Proposition \ref{V.first Matching Proposition}, we get
\begin{equation*}
	\begin{aligned}
		\left|T_*^n \lambda_{1}-T_*^n \lambda_{2}\right|&\leq 2Dn^{-a+1}+DD_{5}\sum_{i=1}^{\infty}i^{-\zeta \tau}  (i+1)\left( \frac{n}{i+1}\right)^{-a+1}\\
		&\leq2Dn^{-a+1}+K_{3}^{\prime}DD_{5}n^{-a+1}\sum_{i=1}^{\infty}i^{-\zeta \tau}  i^{a}, \text{ for some }  K_{3}^{\prime}>0.
	\end{aligned}
\end{equation*}
First we want to estimate the second term. For that we consider the following cases:
\begin{itemize}
\item $\tau=\frac{a+1}{\zeta}$. The sum is
$$
\sum_{i=1}^{n} i^{-\zeta \tau+a}=\sum_{i=1}^{n} i^{-1} \leq 1+\int_1^{n} x^{-1} d x=1+\log(n)=\mathcal{O}(\log n).
$$
So the whole term is $\mathcal{O}\left(n^{1-a} \log n\right)$.
\item $\tau>\frac{a+1}{\zeta}$. Here, $a-\zeta \tau<-1$, so the sum is bounded from above independently of $n$, and the whole term is $\mathcal{O}\left(n^{1-a}\right)$.
\item $\frac{2}{\zeta}<\tau<\frac{a+1}{\zeta}$. The sum is of order $n^{a+1-\zeta \tau}$, and so the whole term is $\mathcal{O}\left(n^{2-\zeta \tau}\right).$
\end{itemize}
Consequently, 
	\begin{itemize}
	\item[i)] if $\tau=\frac{a+1}{\zeta}$, then $\left|T_*^n \lambda_{1}-T_*^n \lambda_{2}\right|=\mathcal{O}\left(n^{1-a} \log n\right)$;
	\item[ii)]  otherwise, $\left|T_*^n \lambda_{1}-T_*^n \lambda_{2}\right|=\mathcal{O}\left(\max \left(n^{1-a}, n^{2-\zeta \tau}\right)\right)$.
\end{itemize}

\subsection{Exponential Return time}\label{Exponentional Return time}

If $m\{R>$ $n\}=\mathcal{O}\left(e^{-cn}\right)$, for some $ c>0 $, then by Lemma \ref{estimate for simulatainious retun time} we have
\begin{equation}\label{1V.Final estimatr of simulatainious RT}
	P\{S>n\}=\mathcal{O}(e^{-c_{0} n}).
\end{equation} 
By similar arguments used to estimate $P\{S>n\}$, we have 
\begin{equation}\label{2V.estimatr of simulatainious RT with mxm}
	(m\times m)\left\{S>\frac{n}{i+1}\right\}=\mathcal{O}(e^{\frac{-c_{0}n}{i+1}}).
\end{equation}

Class $(V1, \tau)$: Assume that $v_i(\Phi)=\mathcal{O}\left(\theta_1^i\right)$ for some $\theta_1<1$. By Lemma \ref{ Chosing a sequence of epsilon(i)}, we have 
\begin{equation}\label{V1}
\prod_{j=1}^i\left(1-\frac{\varepsilon_j}{D_{4}}\right)=\mathcal{O}\left(\theta_2^i\right)
\end{equation}
for some $0<\theta_2<1$. 

By using \eqref{1V.Final estimatr of simulatainious RT}, \eqref{2V.estimatr of simulatainious RT with mxm} and \eqref{V1} in the Proposition \ref{V.first Matching Proposition}, we get
$$
\left|T_*^n \lambda_{1}-T_*^n \lambda_{2}\right|\leq 2De^{-c_{0} n}+DD_{5}\sum_{i=1}^{n}\theta_2^i(i+1)\left(e^{-c_{0} (\frac{n}{i+1})}\right)=\mathcal{O}(e^{-c^{\prime} n}) \text{ for some } c^{\prime}>0.
$$
By similar arguments as used in the above class, we have the following estimates for other classes.

Class  $(V2, \tau)$: Assume that $v_i(\Phi)=\mathcal{O}\left(e^{-i^\tau}\right)$, for some $\tau \in(0,1)$. We have
$$
\left|T_*^n \lambda_{1}-T_*^n \lambda_{2}\right|\leq 2De^{-c_{0} n}+DD_{5}\sum_{i=1}^{n} e^{-\zeta i^\tau} (i+1)\left(e^{-c_{0} (\frac{n}{i+1})}\right) =\mathcal{O}(e^{-n^{\tau^\prime}}),
$$
for $\tau^\prime< \tau$.

Class  $(V3, \tau)$: Assume that $v_i(\Phi)=\mathcal{O}\left(e^{-(\log i)^\tau}\right)$, for some $\tau>1$, then
$$
\left|T_*^n \lambda_{1}-T_*^n \lambda_{2}\right|\leq 2De^{-c_{0} n}+DD_{5}\sum_{i=1}^{n}  e^{-\zeta(\log i)^\tau} (i+1)\left(e^{-c_{0} (\frac{n}{i+1})}\right)=\mathcal{O}(e^{-(\log(n))^{\tau^\prime}}), 
$$
for $\tau^\prime< \tau$.

Class  $(V4, \tau)$: Assume that $v_i(\Phi)=\mathcal{O}\left(i^{-\tau}\right)$, for some $\tau>\frac{1}{\zeta}$. We have 
$$
\left|T_*^n \lambda_{1}-T_*^n \lambda_{2}\right|\leq 2De^{-c_{0} n}+DD_{5}\sum_{i=1}^{n}  i^{-\zeta \tau} (i+1)\left(e^{-c_{0} (\frac{n}{i+1})}\right)=\mathcal{O}(n^{1-\tau\zeta}).
$$

Let us conclude the proof of Proposition \ref{V.Convergence to equilibrium}. We notice in the proof of Lemma \ref{V.lemma:phi*}, that $ v_n(\varphi^{\ast}\zeta_{0})\leq 3C_{\zeta_{0}}\beta^{n}+C_{0}v_{n}(\varphi)$, this implies that $ v_{n}(\Phi)\leq 3C_{\zeta_{0}}\beta^{n}+C_{0}v_{n}(\varphi)$. Now Taking $ \lambda_{1}=\lambda $ and $ \lambda_{2}=\nu $ in Sections  \ref{Polynomial Return time} and \ref{Exponentional Return time}, we get the required estimates for $ |T_*^n \lambda-\nu| $. We want to check the dependence of the constants. From \eqref{V.L final log density} $ C_{\Phi}=3K_{1}^{\prime}+2C$. This means that the constant $ C_{\Phi} $ does not depend on $ \varphi $.
Therefore the constant $ \zeta=\frac{\bar{\delta}}{K} $ will only depend on $ v_{0}(\Phi) $, which can be bounded with $ v_{0}(\varphi),$ as $v_0(\Phi)\leq 3C_{\zeta_{0}}+C_0v_0(\varphi)$.

\bibliographystyle{alpha}

\end{document}